\documentclass[12pt]{amsart}
\usepackage{amssymb, amsthm, amsmath}

\newtheorem{question}{Question}[section]
\newtheorem{theorem}[question]{Theorem}
\newtheorem{lemma}[question]{Lemma}
\newtheorem{corollary}[question]{Corollary}
\newtheorem{definition}[question]{Definition}

\newtheorem{proposition}[question]{Proposition}
\newtheorem{example}[question]{Example}

\title{Selective versions of chain condition-type properties}

\author{Leandro F. Aurichi}
\address{Instituto de Ciencias Matematicas e de Computacao (ICMC-USP) \\ Universidade de Sao Paulo \\  Avenida Trabalhador Sao-carlense, 400 - Centro\\ 13566--590 Sao Carlos  - SP\\ Brazil}
\email{aurichi@icmc.usp.br}

\author{Santi Spadaro}
\address{Instituto de Matematica e Estatistica (IME-USP) \\ Universidade de Sao Paulo \\ 
Rua do Matao, 1010 - Cidade Universitaria \\  05508-090 Sao Paulo - SP \\ Brazil}
\email{santidspadaro@gmail.com}

\author{Lyubomyr Zdomskyy}
\address{Kurt Godel Research Center for Mathematical Logic\\ University of Vienna\\ W\"ahringer Strasse 25\\ A-1090 Wien\\ Austria}
\email{lzdomsky@gmail.com}



\begin{document}

\begin{abstract}
We study selective and game-theoretic versions of properties like the ccc, 
weak Lindel\"ofness and separability, giving various characterizations of them 
and exploring connections between these properties and some classical cardinal 
invariants of the continuum.

\end{abstract}

\maketitle

\section{Introduction}

Chain conditions provide a measure of how \emph{small} is a space, from a  topological point of view. For example, a space has the \emph{countable chain condition} (ccc) if it does not contain an uncountable family of pairwise disjoint non-empty open sets. This is one of the weakest chain conditions one can put on a space, while separability may be considered one of the strongest. Of course, every separable space satisfies the ccc.

Todorcevic \cite{T} surveys a wealth of chain conditions that are between the ccc and separability, elaborating on their \emph{classifying power}, in the sense that the better the space, the greater number of chain conditions it identifies. A few examples of this phenomenon are Knaster's result that separability and \emph{Knaster's property} (that is, every uncountable family of open sets contains an uncountable family where each pair of elements meets) are equivalent on ordered continua and Shapirovskii's result that separability and the \emph{Shanin condition} (that is, point-countable families of open sets are countable) are equivalent for compact spaces of countable tightness. These are all ZFC results, but the picture is even clearer in certain models of set theory: for example, the ccc and separability are equivalent for linearly ordered spaces under $MA_{\omega_1}$. Some of these characterizations even offer topological equivalents of certain combinatorial principles. Let us just mention Todorcevic and Velickovic's result that $MA_{\omega_1}$ is equivalent to the statement that every compact first-countable ccc space is separable.

In more recent times the framework of selection principles in mathematics and topological games has been applied to chain conditions offering more examples of the phenomenon underlined by Todorcevic. A typical selective chain condition is the one considered by the first-named author in \cite{A} which states that one can diagonalize a family with a dense union from a countable sequence of maximal pairwise disjoint families of non-empty open sets. Daniels, Kunen and Zhou introduced a game-theoretic strengthening of this property in \cite{DKZ} by using a two-player game where each player plays an inning per natural number and at a given inning, the first player picks a maximal pairwise disjoint open family, while the second player picks an open set from it. The second player wins if the set of all open sets he picked has dense union. If the second player has a winning strategy in this game on a given space $X$, then $X$ is selectively ccc, which in turn implies that $X$ is ccc. Unfortunately, there are countable spaces failing the selective versions of the ccc (see \cite{A}), so separability alone does not appear to play any role in this context. What should take the role of separability when dealing with selective chain conditions is countable $\pi$-weight. Recall that a \emph{$\pi$-base} for a space $X$ is a family $\mathcal{P}$ of non-empty open sets such that for every non-empty open $U \subset X$ there is $P \in \mathcal{P}$ such that $P \subset U$. The $\pi$-weight of a space is then the minimum cardinality of a $\pi$-base. It is easy to see that in every space with a countable $\pi$-base the second player has a winning strategy in the ccc game. The second-named author proved in \cite{BS} that this is actually equivalent to having a countable $\pi$-base for spaces with a countable local $\pi$-base at every point. A selective version and a game-theoretic version of separability were defined by Scheepers in \cite{SchSep} in a similar way. The former turns out to be equivalent to a countable $\pi$-base for compact spaces and the latter is equivalent to a countable $\pi$-base for all regular spaces. These results seem to suggest that the same classifying ability of traditional chain condition is shared by their selective versions, as long as one takes countable $\pi$-weight as the \emph{ultimate selective chain condition}.

Even though they measure the \emph{topological smallness} of a space, chain conditions seldom put any bound on its cardinality. Indeed, while a separable regular space cannot have cardinality larger than $2^{\mathfrak{c}}$, there are ccc spaces of arbitrarily large cardinality: it suffices to consider the Cantor Cube $2^\kappa$, where $\kappa$ is any cardinal. Nonetheless, chain conditions feature prominently in a few classical cardinal inequalities.  A result of Hajnal and Juh\'asz states that the cardinality of a ccc first countable space does not exceed the continuum. An interesting partial generalization of this theorem, that also generalizes Arhangel'skii's theorem on the cardinality of first-countable compacta, is due to Bell, Ginsburg and Woods. They proved that the cardinality of a first-countable weakly Lindel\"of normal space does not exceed the continuum. Weakly Lindel\"of means that every open cover has a countable subcollection with dense union, a condition which is easily seen to be satisfied by all ccc spaces. The question of whether normality can be relaxed to regularity in this result has remained open, but Angelo Bella gave a partial answer to it in \cite{BS} by considering the natural game-theoretic strengthening of the weak Lindel\"of property. In the third section of our paper we prove this game is equivalent to a sort of dual of Berner and Juh\'asz's classical point-picking game and exploit this equivalence to give a short proof to Bella's Theorem.

Another reason for our interest in game-theoretic strengthenings of chain conditions is that they provide an unexpected ZFC partial positive answer to the old problem of the productivity of the ccc. It was already known by Kurepa that the square of a Suslin Line is not ccc. Thus, consistently, the countable chain condition is not productive. Both Knaster's property and the Shanin condition are productive, and one can use these results along with the above mentioned equivalences to prove that under $MA_{\omega_1}$ the ccc is productive. In \cite{DKZ} Daniels, Kunen and Zhou proved in ZFC  that if player II has a winning strategy in the ccc game on every factor of an arbitrary product then it also has it on the full product.

In section 2 we deal with selective properties. In particular, we characterize the selective ccc on Pixley-Roy hyperspaces and we give a consistent topological characterization of $cov(\mathcal{M})$ by means of the weak Rothberger property.

In section 3 we concentrate on game versions. We prove preservation results in finite unions and products and give a characterization of the weak Rothberger game that we then exploit to prove cardinal inequalities in topological spaces. We then construct counterexamples proving the sharpness of our inequalities.

Given a space $X$ with topology $\tau$, we fix some notation about families of open covers and families of subsets of $X$.
\begin{itemize}
\item $\mathcal{O}=\{\mathcal{U}: \mathcal{U} \subset \tau \wedge \bigcup \mathcal{U}=X \}$.
\item $\mathcal{O}_D=\{\mathcal{U}: \mathcal{U} \subset \tau \wedge \overline{\bigcup \mathcal{U}}=X \}$.
\item $\Omega=\{\mathcal{U}: \mathcal{U} \subset \tau \wedge (\forall F \in [X]^{<\omega} \exists O \in \mathcal{U}: F \subset O \}$.
\item $\mathcal{D}= \{D \subset X: \overline{D}=X\}$.
\item $\mathcal{D}_O=\{O \in \tau: \overline{O}=X \}$.

\end{itemize}

An element of $\Omega$ is usually known as an \emph{$\omega$-cover} of $X$.

Let's recall some basic selection principles and two-person infinite games we will deal with in our paper. Let $\mathcal{A}, \mathcal{B} \subset \mathcal{P}(X)$.

\begin{itemize}
\item We say that $X$ satisfies $S^\kappa_1(\mathcal{A}, \mathcal{B})$ if for every sequence $\{A_\alpha: \alpha< \kappa \} \subset \mathcal{A}$, we can choose $B_\alpha \in A_\alpha$ such that $\{B_\alpha: \alpha < \kappa \} \in \mathcal{B}$.
\item We say that $X$ satisfies $S^\kappa_{fin}(\mathcal{A}, \mathcal{B})$ if for every sequence $\{A_\alpha: \alpha < \kappa \} \subset \mathcal{A}$, we can choose $B_\alpha \in [A_\alpha]^{<\omega}$ such that $\bigcup \{B_\alpha: \alpha < \kappa \} \in \mathcal{B}$.
\item We denote by $G^\kappa_1(\mathcal{A}, \mathcal{B})$ the two-person game in $\kappa$ many innings such that, at inning $\alpha<\kappa$, player one picks $A_\alpha \in \mathcal{A}$ and player two picks $B_\alpha \in A_\alpha$. Player two wins if $\{B_\alpha: \alpha < \kappa \} \in \mathcal{B}$.
\item We denote by $G^\kappa_{fin}(\mathcal{A}, \mathcal{B})$ the two-person game in $\kappa$ many innings such that, at inning $\alpha<\kappa$, player one picks $A_\alpha \in \mathcal{A}$ and player two picks $B_\alpha \in [A_\alpha]^{<\omega}$. Player two wins if $\bigcup \{B_\alpha: \alpha < \kappa \} \in \mathcal{B}$.
\end{itemize}

The properties $S^\omega_1(\mathcal{O}, \mathcal{O})$ and $S^\omega_{fin}(\mathcal{O}, \mathcal{O})$ are now known as \emph{Rothberger} and \emph{Menger} respectively. Moreover, we recall that $X$ satisfies $S^\omega_1(\Omega, \Omega)$ if and only if every finite power of $X$ is Rothberger.

\section{Selective versions}

In \cite{SchSep}, Marion Scheepers defined a natural selective version of separability which is now known as $R$-separability (see \cite{BBM}). The $R$ in the name comes, of course, from Rothberger.

\begin{definition}
A space is \emph{$R$-separable} if it satisfies $S^\omega_1(\mathcal{D}, \mathcal{D})$. In other words, for every sequence $\{D_n: n < \omega \}$ of dense sets there is a point $x_n \in D_n$, for every $n<\omega$ such that $\{x_n: n < \omega \}$ is dense in $X$.
\end{definition}

$R$-separability is much stronger than separability. For example, it implies that every dense set is separable, and by a result of Juh\'asz and Shelah \cite{JS}, this is equivalent to countable $\pi$-weight in the realm of compact spaces. Moreover, there are even examples of countable spaces that are not $R$-separable (see \cite{BBMadd} and \cite{SSS}).

In \cite{A} the first named author introduced the following weakening of $R$-separability, following a suggestion of Sakai.

\begin{definition}
We say that the space $X$ has property $S$ if it satisfies $S^\omega_1(\mathcal{D}_O, \mathcal{D})$, that is, for every sequence $\{O_n: n < \omega \}$ of open dense sets we can pick points $x_n \in O_n$ such that $\{x_n: n < \omega \}$ is dense.
\end{definition}

One of the most interesting features about property $S$ is that it lies strictly between $R$-separability and a natural selective version of the countable chain condition that was introduced by Scheepers in \cite{Schccc}.

\begin{definition}
We say that $X$ is \emph{selectively ccc} if it satisfies $S^\omega_1(\mathcal{O}_D, \mathcal{O}_D)$. In other words, for every sequence $\{\mathcal{U}_n: n < \omega \}$ of open families, such that $\bigcup \mathcal{U}_n$ is dense for every $n<\omega$ we can pick an open set $U_n \in \mathcal{U}_n$ such that $\bigcup \{U_n: n <\omega \}$ is dense.
\end{definition}

Obviously, every space having property S is selectively ccc, but the converse does not hold, since $2^\kappa$ is selectively ccc for every cardinal $\kappa$  (see \cite{A}), but fails to have property $S$ for $\kappa > \mathfrak{c}$ because it's not even separable then.

Let $X$ be a space and set $PR(X)=[X]^{<\omega}$. There is a natural topology on $PR(X)$ called the Pixley-Roy topology. A basic open neighbourhood of $F \in PR(X)$ is a set of the form $[F,U]=\{G \in PR(X): F \subset G \subset U \}$, where $F \subset U$ and $U \subset X$ is open in the topology on $X$.

While other selective properties have been characterized on $PR(X)$ (see, for example, \cite{Sak} and \cite{BCPT}), the problem of characterizing the selective ccc of $PR(X)$ has remained open. We introduce a new selection principle to provide such a characterization and then reduce the selective ccc of the Pixley-Roy hyperspace of a separable metrizable space $X$ to a well-known selective covering property of $X$.

\begin{definition}
Let $(X,\tau)$ be a topological space, and $\mathcal{F} \subset \{(F,U) \in [X]^{<\omega} \times \tau: F \subset U \}$. We call $\mathcal{F}$ an \emph{$\omega$-double cover} if for every pair $(G,V) \in [X]^{<\omega} \times \tau$ such that $G \subset V$ there is $(F,U) \in \mathcal{F}$ such that $F \subset V$ and $G \subset U$. The family of all $\omega$-double covers will be indicated with $\Omega^2$.
\end{definition}

\begin{proposition} \label{simpleprop}
$S^\omega_1(\Omega^2, \Omega^2) \Rightarrow S^\omega_1(\Omega, \Omega)$.
\end{proposition}

\begin{proof}
Just note that if $\mathcal{U}_n$ is an $\omega$-cover then $\{(\emptyset,U): U \in \mathcal{U}_n \}$ is an $\omega$-double cover.
\end{proof}

\begin{theorem}
$PR(X)$ has $S^\omega_1(\mathcal{O}_D, \mathcal{O}_D)$ if and only if $X$ has $S^\omega_1(\Omega^2, \Omega^2)$.
\end{theorem}

\begin{proof}
For the direct implication, let $\{\mathcal{F}_n: n <\omega \} \subset \Omega^2$. Now let $\mathcal{O}_n=\{[F,U]: (F,U) \in \mathcal{F}_n \}$. Then $\{\mathcal{O}_n: n < \omega \} \subset \mathcal{O}_D(PR(X))$. By the $S^\omega_1(\mathcal{O}_D, \mathcal{O}_D)$ property of $PR(X)$ we can find $[F_n, U_n] \in \mathcal{O}_n$ such that $\{[F_n, U_n]: n < \omega \} \subset \mathcal{O}_D$. We then have that $\{(F_n,U_n): n < \omega \} \in \Omega^2$. Indeed, let $(F,U) \in  \{(F,U) \in [X]^{<\omega} \times \tau: F \subset U \}$. Then $[F,U] \cap [F_n,U_n] \neq \emptyset$ for some $n$. So there is $H$ extending both $F$ and $F_n$ such that $H \subset U_n$ and $H \subset U$, and that implies $F \subset U_n$ and $F_n \subset U$.

Vice versa, suppose $X$ has $S^\omega_1(\Omega^2, \Omega^2)$ and let $\{\mathcal{U}_n: n <\omega \}$ be a sequence of open families with dense union in $PR(X)$. Without loss of generality we can assume that $\mathcal{U}_n$ is made up of basic open sets. Let now $\mathcal{F}_n=\{(F,U): [F,U] \in \mathcal{U}_n\}$. Then $\mathcal{F}_n \in \Omega^2$, for every $n<\omega$. Hence we can find $(F_n,U_n) \in \mathcal{F}_n$ such that $\{(F_n, U_n): n < \omega \} \in \Omega^2$. Now by the same argument $\{[F_n,U_n]: n < \omega \}$ has dense union in $PR(X)$.
\end{proof}

\begin{corollary}
If $PR(X)$ satisfies $S^\omega_1(\mathcal{O}_D, \mathcal{O}_D)$ (that is, $PR(X)$ is selectively ccc) then $X$ satisfies $S^\omega_1(\Omega, \Omega)$.
\end{corollary}

\begin{corollary} \label{opensub}
If $X$ satisfies $S^\omega_1(\Omega^2, \Omega^2)$ and $U \subset X$ is open then $U$ satisfies $S^\omega_1(\Omega^2, \Omega^2)$.
\end{corollary}

\begin{proof}
If $X$ satisfies $S^\omega_1(\Omega^2, \Omega^2)$ then $PR(X)$ is selectively ccc. Now $PR(U)=[\emptyset, U]$ is an open subspace of a selectively ccc space and hence it's selectively ccc. But then $U$ satisfies $S^\omega_1(\Omega^2, \Omega^2)$.
\end{proof}

\begin{theorem}
Let $X$ be a second-countable space. Then $X$ is $S^\omega_1(\Omega^2, \Omega^2)$ if and only if $X$ is $S^\omega_1(\Omega, \Omega)$.
\end{theorem}

\begin{proof}
The direct implication is clear by Proposition $\ref{simpleprop}$. For the converse implication, let $\{B_n: n < \omega \}$ be a countable base for $X$ which is closed under finite unions. Let $\{I_k: k < \omega \}$ be a partition of $\omega$ into infinite sets. Let $\{\mathcal{U}_n: n < \omega \} \subset \Omega^2$. We can assume without loss of generality that for every $n<\omega$ and for every $(F,U) \in \mathcal{U}_n$ there is $k<\omega$ such that $U \subset B_k$. Moreover, we can assume that for every $n< \omega$, and for every $(F,U) \in \mathcal{U}_n$, if $V$ is an open subset of $X$ such that $F \subset V \subset U$ then $(F,V) \in \mathcal{U}_n$. For every $n \in I_k$, define families $\mathcal{V}_n$ as follows:

$$\mathcal{V}_n=\{U: (F,U) \in \mathcal{U}_n \wedge U \subset B_k \}$$

Then $\mathcal{V}_n$ is an $\omega$-cover of $B_k$ and hence we can pick an element $U_n \in \mathcal{V}_n$ for every $n \in I_k$ such that $\{U_n: n \in I_k \}$ is an $\omega$-cover of $B_k$. Let $F_n \in [X]^{<\omega}$ be such that $F_n \subset U_n$ and $(F_n, U_n) \in \mathcal{U}_n$. We claim that $\mathcal{V}=\{(F_n,U_n): n < \omega \}$ is an $\omega$-double cover for $X$. Indeed, let $(G,V)$ be any pair, where $G \in [X]^{<\omega}$, $V \subset X$ is open and $G \subset V$. Let $k<\omega$ be such that $G \subset B_k \subset V$. Then we can find a $n \in I_k$ such that $G \subset U_n$. Now $F_n \subset U_n \subset B_k \subset V$ and hence $\mathcal{V}$ is actually an $\omega$-double cover for $X$.

\end{proof}

\begin{corollary}
Let $X$ be a separable metrizable space. Then $PR(X)$ is selectively ccc if and only if every finite power of $X$ is Rothberger.
\end{corollary}

The weak Lindel\"of property is a covering property that may be considered a chain condition, since it is a consequence of the ccc. We finish this section by considering the natural question of when a weakly Lindel\"of space satisfies the selective version of weak Lindel\"ofness, that is, $S^\omega_1(\mathcal{O}, \mathcal{O}_D)$. As in \cite{BPS}, we will call this property the \emph{weak Rothberger property}.

\begin{theorem}
Let $X$ be a weakly Lindel\"of space such that $\pi w(X) <cov(\mathcal{M})$. Then $X$ is weakly Rothberger.
\end{theorem}

\begin{proof}

Recall that the space $\omega^\omega$ with its usual topology is homeomorphic to the irrationals and $cov(\mathcal{M})$ can be characterized as the least cardinal of a cover of the irrationals by nowhere dense sets.

Let $\{B_\alpha: \alpha < \kappa \}$ enumerate a $\pi$-base of $X$, for some $\kappa < cov(\mathcal{M})$. Let $\{\mathcal{U}_n: n < \omega \}$ be a sequence of open covers. Since $X$ is weakly Lindel\"of we can find a countable subcollection $\{U^n_k: k < \omega \}$ of $\mathcal{U}_n$ such that $\bigcup_{k<\omega} U^n_k$ is dense. Define $N_\alpha$ to be the following set: $$N_\alpha=\{f \in \omega^\omega: (\forall n)(B_\alpha \cap U^n_{f(n)}=\emptyset) \}$$

\noindent {\bf Claim:} The set $N_\alpha$ is nowhere dense in $\omega^\omega$ for every $\alpha < \kappa$.

\begin{proof}[Proof of Claim] It will suffice to prove that for every $\alpha < \kappa$, the set $\omega^\omega \setminus \overline{N_\alpha}$ is dense in $\omega^\omega$. Let $[\sigma]$ be a basic open set in $\omega^\omega$, where $\sigma \in \bigcup_{n<\omega} \omega^n$ and $[\sigma]=\{f \in \omega^\omega: f \supset \sigma \}$. Let $k=dom(\sigma)$ and pick $j<\omega$ such that $B_\alpha \cap U^k_j \neq \emptyset$. Let $\tau=\sigma \cup \{(k,j)\}$ and note that $[\tau] \subset [\sigma]$ and any function in the open set $[\tau]$ misses $N_\alpha$.
\renewcommand{\qedsymbol}{$\triangle$}
\end{proof}

Since $\kappa < cov(\mathcal{M})$ we can pick $f \notin \bigcup_{\alpha < \kappa} N_\alpha$. Then $\{U^n_{f(n)}: n < \omega \}$ is the selection showing that $X$ is weakly Rothberger. 
\end{proof}

\begin{theorem}
Under $cov(\mathcal{M})<\mathfrak{b}$ there is a compact space which is not weakly Rothberger and has $\pi$-weight $cov(\mathcal{M})$.
\end{theorem}

\begin{proof}
Assume that $cov(\mathcal{M})<\mathfrak b$. We claim that $\beta (cov(\mathcal{M}))$ is the required example.

We say that $F \subset \omega^\omega$ satisfies property (P) if for every $g \in \omega^\omega$  there is $f \in F$ such that $f(n) \neq g(n)$ for every $n<\omega$.

\noindent {\bf Claim.} There is a family $F \subset \omega^\omega$ of cardinality $cov(\mathcal{M})$ satisfying property (P) such that $f(n) < b(n)$ for some $b \in \omega^\omega$ and every $f \in F$.

\begin{proof}[Proof of Claim]
By Lemma 2.4.2 of \cite{BaJu} there is a family $F' \subset \omega^\omega$ of cardinality $cov(\mathcal{M})$ satisfying (P). Let $\{h_\alpha: \alpha < cov(\mathcal{M}) \}$ be an enumeration of $F'$.

Since $cov(\mathcal{M}) < \mathfrak{b}$, we can fix a function $b \in \omega^\omega$ such that $h_\alpha <^* b$, for every $\alpha < cov(\mathcal{M})$. Now, for every $\alpha < cov(\mathcal{M})$ there is $n_\alpha < \omega$ such that $h_\alpha(n) < b(n)$ for every $n \geq n_\alpha$. For every $\alpha < cov(\mathcal{M})$, $n<\omega$ and $i<b(n)$ define:
$$h^i_\alpha(n)=
\begin{cases}
i & \mbox{if } n<n_\alpha \\
h_\alpha(n) & \mbox{if } n \geq n_\alpha
\end{cases}
$$

It's easy to see that $F=\{h^i_\alpha: i < b(n), \alpha < cov(\mathcal{M}) \}$ is a subfamily of $\omega^\omega$ satisfying (P) and such that $h(n) < b(n)$, for every $h \in F$.
\renewcommand{\qedsymbol}{$\triangle$}
\end{proof}

Let $\{f_\alpha: \alpha < cov(\mathcal{M})\}$ be an enumeration of $F$ and for every $n$ let us consider the following finite clopen cover of
$\beta(cov(M))$: $\mathcal U_n=\{U^n_k:k<b(n)\}$, where
$U^n_k = \beta(\{\xi<cov(\mathcal{M}) : f_\xi(n)=k \})$ and $\beta(A)$ is the set of all
ultrafilters on $cov(\mathcal{M})$ containing $A$. 
We claim that the sequence $\{ \mathcal U_n:n\in\omega \}$
witnesses that $\beta(cov(\mathcal{M}))$ fails to have $S^\omega_1(\mathcal O, \mathcal O_D)$.
Indeed, suppose that $\bigcup_{n\in\omega}U^n_{g(n)}$ is dense in $\beta(cov(\mathcal{M}))$
for some $g\in \omega^\omega$. Then, in particular, $cov(\mathcal{M})\subset \bigcup_{n\in\omega}U^n_{g(n)}$.

But since $F$ satisfies property (P), there exists $\xi<cov(\mathcal{M})$ such that
$f_\xi(n)\neq g(n)$ for all  $n$, which means that $\xi\not\in U^n_{g(n)}$ for all $n$, and that is a contradiction.
\end{proof}

\begin{corollary}
Assume $cov(\mathcal{M})<\mathfrak{b}$. Then $cov(\mathcal{M})$ can be characterized as the minimum $\pi$-weight of a non-weakly Rothberger weakly Lindel\"of space.
\end{corollary}

\begin{question}
Is it true in ZFC that $cov(\mathcal{M})$ is the minimum $\pi$-weight of a non-weakly Rothberger weakly Lindel\"of space?
\end{question}

\section{Game versions}
The property $S$ defined in the previous section has a natural game version.

\begin{definition}
We say that $X$ has property $S^+$ if the second player has a winning strategy in the game $G^\omega_1(\mathcal{D}_O, \mathcal{D})$. This is the two person game of countable length where at inning $n$ player one picks a dense open set $O_n$ and player two picks a point $x_n \in O_n$. Player two wins if the set of all points he picked is dense in $X$.
\end{definition}

\begin{lemma} \label{openlemma}
Property $S^+$ is hereditary for open sets.
\end{lemma}

\begin{proof}
Let $X$ be a space with property $S^+$ and $U \subset X$ be a non-empty open subset. Let $\tau$ be a winning strategy for player II in $G^\omega_1(\mathcal{D}_O, \mathcal{D})$ on $X$. Let $\sigma$ be the strategy assigning to the open dense subset $O$ of $U$ the point $\tau(O \cup Int(X \setminus W))$ if this last point is in $O$, and any point of $O$ otherwise. Then $\sigma$ is a winning strategy for player II in $G^\omega_1(\mathcal{D}_O, \mathcal{D})$ on $O$.
\end{proof}

The following fact is also clear.

\begin{lemma}
If $D$ is dense in $X$ and $D$ has property $S^+$ then $X$ also has property $S^+$.
\end{lemma}

\begin{proposition}
Property $S^+$ is preserved by finite unions.
\end{proposition}

\begin{proof}
Once this is proved for the union of two spaces, the result will follow easily by induction, so let $X$ be a topological space and $A$ and $B$ be subspaces with property $S^+$ such that $X=A \cup B$. If $Int(\overline{A}) \cap A=\emptyset$ then $B$ is dense in $X$ and we are done. Similarly, if $Int(\overline{B}) \cap B = \emptyset$ we are done. So we can assume that $Int(\overline{A}) \cap A$ and $Int(\overline{B}) \cap B$ are both non-empty subsets of $X$. By Lemma $\ref{openlemma}$ we can fix a winning strategy $\sigma_A$ for player two on $Int(\overline{A}) \cap A$ and a winning strategy $\sigma_B$ for player two on $Int(\overline{B}) \cap B$ in the game $G^\omega_1(\mathcal{D}_O, \mathcal{D})$. 
The set $Int(\overline{A}) \cup Int(\overline{B})$ is a dense open subset of $X$. Note that if $U$ is dense open in $X$ then $U \cap Int(\overline{A}) \cap A$ is dense open in $Int(\overline{A}) \cap A$ and $U \cap Int(\overline{B}) \cap B$ is dense open in $Int(\overline{B}) \cap B$. We are now going to define a strategy $\sigma$ for player two on $A \cup B$.

Suppose that in the first $n$ innings, player two played the following sequence of dense open sets $(U_i: i \leq n)$. 

If $n=2k$ for some $k<\omega$, we let
$$\sigma((U_i: i \leq n))=\sigma_A((U_{2i} \cap Int(\overline{A}) \cap A: i \leq k))$$

If $n=2k+1$ for some $k<\omega$, we let 
$$\sigma((U_i: i \leq n))=\sigma_B((U_{2i+1} \cap Int(\overline{B}) \cap B: i \leq k))$$

Note now that, by the definition of $\sigma_A$ and $\sigma_B$ we have that:

$$\overline{\{\sigma((U_i: i \leq n)): n < \omega\}}=$$
$$=\overline{\{\sigma_A((U_{2i}: i \leq k)): k < \omega \}} \cup \overline{\{\sigma_B((U_{2i+1}: i \leq k)): k <\omega \}}=$$
$$=\overline{Int(\overline{A}) \cap A} \cup \overline{Int(\overline{B}) \cap B}=X$$

So $\sigma$ is a winning strategy for player two in $G^\omega_1(\mathcal{D}_O, \mathcal{D}))$ on $X$ and we are done.
\end{proof}

In a similar way, one can prove the following propositions:

\begin{proposition}
If player two has a winning strategy in $G^\omega_{fin}(\mathcal{D}_O, \mathcal{D})$ on $A_i$, for every $i \leq n$ then he also has a winning strategy in that game on $A_1 \cup A_2 \cup \dots \cup A_n$.
\end{proposition}

\begin{proposition}
The properties $S^\omega_1(\mathcal{D}_O, \mathcal{D})$ and $S^\omega_{fin}(\mathcal{D}_O, \mathcal{D}))$ are preserved by finite unions.
\end{proposition}

The natural game version of the weak Rothberger property, introduced in the end of the previous section is the game $G^\omega_1(\mathcal{O}, \mathcal{O}_D)$. This is tightly connected to the dual version of the point-picking game studied by Berner and Juh\'asz in \cite{BJ}.

\begin{definition}
Let $G^p_o(\kappa)$ be the following game. At inning $\alpha < \kappa$, player one picks a point $x_\alpha \in X$ and player two picks an open set $U_\alpha$ such that $x_\alpha \in U_\alpha$. Player one wins if  and only if $\bigcup_{\alpha < \kappa} U_\alpha$ is dense in $X$.
\end{definition}

We may call this the \emph{open-picking game}. We prove that the game $G^\kappa_1(\mathcal{O}, \mathcal{O}_D)$ and the game $G^p_o(\kappa)$ are \emph{dual}, in the following sense.

\begin{theorem} \label{thmequiv} {\ \\}
\begin{enumerate}
\item \label{playerone} Player one has a winning strategy in $G^p_o(\kappa)$ if and only if player two has a winning strategy in $G^\kappa_1(\mathcal{O}, \mathcal{O}_D)$.
\item \label{playertwo} Player two has a winning strategy in $G^p_o(\kappa)$ if and only if player one has a winning strategy in $G^\kappa_1(\mathcal{O}, \mathcal{O}_D)$.
\end{enumerate}
\end{theorem}

\begin{proof}

The direct implication of ($\ref{playerone}$) is easy to see. Indeed, let $\tau$ be a winning strategy for player one in $G^p_o(\kappa)$. Given an open cover $\mathcal{U}$ let $\sigma((\mathcal{U}))$ be any open set $O \in \mathcal{U}$ such that $\tau(\emptyset) \in O$. Assuming we have defined $\sigma$ for the first $\alpha$ many innings, and $\{\mathcal{O}_\beta: \beta \leq \alpha \}$ be a sequence of open covers, let $\sigma((\mathcal{O}_\beta: \beta \leq \alpha))$ be any open set $O \in \mathcal{O}_\alpha$ such that $\tau((\sigma(\mathcal{O}_\gamma: \gamma \leq \beta): \beta  < \alpha)) \in O$. Then $\sigma$ is a winning strategy for player two in $G^\kappa_1(\mathcal{O}, \mathcal{O}_D)$. Indeed, let $(\mathcal{O}_0, O_0, \mathcal{O}_1, O_1, \dots \mathcal{O}_\alpha, O_\alpha, \dots)$ be a play, where player two plays according to $\sigma$. Then $\tau((\sigma(\mathcal{O}_\gamma: \gamma \leq \beta): \beta <\alpha)) \in O_\alpha$, and hence $\bigcup \{O_\alpha: \alpha <\kappa \}$ is dense, as $\tau$ is a winning strategy for player one in $G^p_o(\kappa)$.

To prove the converse implication of ($\ref{playerone}$), $\sigma$ be a winning strategy for player two in $G^\kappa_1(\mathcal{O}, \mathcal{O}_D)$ on some space $X$.

\noindent {\bf Claim}. Let $(\mathcal{O}_\alpha: \alpha < \beta)$ be a sequence of open covers. Then there is a point $x \in X$ such that, for every neighbourhood $U$ of $x$ there is an open cover $\mathcal{U}$ with $U=\sigma((\mathcal{O}_\alpha: \alpha < \beta)^\frown (\mathcal{U}))$.

\begin{proof}[Proof of Claim]
Recalling that $\mathcal{O}$ denotes the set of all open covers of $X$, let $\mathcal{V}=\{V$ open: $(\forall \mathcal{U} \in \mathcal{O})(V \neq \sigma((\mathcal{O}_\alpha: \alpha < \beta)^\frown(\mathcal{U})) \}$.  Its definition easily implies that $\mathcal{V}$ cannot be an open cover, and hence there is a point $x \in X \setminus \bigcup \mathcal{V}$. By definition of $\mathcal{V}$ we must have that for every neighbourhood $U$ of $x$ there is an open cover $\mathcal{U}$ such that $U=\sigma((\mathcal{O}_\alpha: \alpha < \beta)^\frown (\mathcal{U}))$ and hence we are done.
\renewcommand{\qedsymbol}{$\triangle$}
\end{proof}

Use the Claim to choose a point $x_0$ such that, for every neighbourhood $U$ of $x_0$ there is an open cover $\mathcal{U}$ with $\sigma(\mathcal{U})=U$ and let $\tau(\emptyset)=x_0$.

Suppose we have defined $\tau$ for the first $\alpha$ many innings. Let now $\{V_\beta: \beta \leq \alpha \}$ be a sequence of open sets and $\{\mathcal{O}_\beta: \beta < \alpha\}$ be a sequence of open covers such that $V_\beta=\sigma((\mathcal{O}_\gamma: \gamma \leq \beta ))$, for every $\beta < \alpha$. Use the claim to choose a point $x_\alpha$ such that, for every open neighbourhood $U$ of $x_\alpha$ there is an open cover $\mathcal{O}$ with $U=\sigma((\mathcal{O}_\beta: \beta < \alpha)^\frown (\mathcal{O}))$ and let $\tau((V_\beta: \beta \leq\alpha))=x_\alpha$.

We now claim that $\tau$ is a winning strategy for player one in $G^p_o(\kappa)$. Indeed, let $(x_0, V_0, x_1, V_1, \dots x_\alpha, V_\alpha, \dots )$ be a play where player one uses strategy $\tau$. Then there must be a sequence of open covers $\{\mathcal{O}_\alpha: \alpha < \kappa \}$ such that $V_\beta=\sigma((\mathcal{O}_\alpha: \alpha < \beta ))$, for every $\beta < \kappa$. Since $\sigma$ is a winning strategy for two in $G^\kappa_1(\mathcal{O}, \mathcal{O}_D)$ then $\bigcup \{V_\alpha: \alpha < \kappa \}$ is dense in $X$ and this proves that $\tau$ is a winning strategy for player one in $G^p_o(\kappa)$.

To prove the direct implication of $(\ref{playertwo})$, let $\tau$ be a winning strategy for player two in $G^p_0(\kappa)$. We define a winning strategy $\sigma$ for player one in $G^\kappa_1(\mathcal{O}, \mathcal{O}_D)$ as follows: in his first move player one plays the open cover $\sigma(\emptyset)= \{\tau(x): x \in X \}$. Assuming we have defined $\sigma$ for every inning $\beta$, with $\beta <\alpha$, let $(U_\beta: \beta < \alpha)$ be a sequence of open sets such that there is a sequence $\{x_\beta: \beta < \alpha \}$ of points with $U_\beta=\tau(x_\gamma: \gamma \leq \beta)$. Then we simply define $\sigma((U_\beta: \beta < \alpha))$ to be $\{\tau((x_\beta: \beta < \alpha)^\frown (x)): x \in X \}$. We claim that $\sigma$ is a winning strategy for player one in $G^\kappa_1(\mathcal{O}, \mathcal{O}_D)$.

Indeed, let $(\mathcal{O}_0, U_0, \dots \mathcal{O}_\alpha, U_\alpha \dots)$ be a play of $G^\kappa_1(\mathcal{O}, \mathcal{O}_D)$ where player one plays according to $\sigma$. Therefore, we can find a sequence $\{x_\alpha: \alpha < \kappa \}$ of points, such that $U_\alpha=\tau((x_\gamma: \gamma < \alpha))$, and hence $\bigcup_{\alpha < \kappa} U_\alpha$ is not dense, since $\tau$ is a winning strategy for player two in $G^p_o(\kappa)$. So $\sigma$ must be a winning strategy for player one in $G^\kappa_1(\mathcal{O}, \mathcal{O}_D)$.

To prove the converse implication of $(\ref{playertwo})$, let $\sigma$ be a winning strategy for player one in $G^\kappa_1(\mathcal{O}, \mathcal{O}_D)$. We will use $\sigma$ to define a winning strategy $\tau$ for player two in $G^p_o(\kappa)$. Given a point $x \in X$, let $\tau((x))$ be any open set $U \in \sigma(\emptyset)$ such that $x \in U$. 

Now suppose $\tau$ has been defined for all sequences of points of ordinal length less than $\alpha$. Given a sequence $\{x_\beta: \beta \leq \alpha \} \subset X$, let $\tau((x_\beta: \beta \leq \alpha))$ be any open set $U \in \sigma((\tau((x_\gamma: \gamma \leq \beta): \beta < \alpha))$ such that $x_\alpha \in U$. We claim that $\tau$ thus defined is a winning strategy for player two in $G^p_o(\kappa)$.

Indeed, let $x_0, U_0, x_1, U_1, \dots x_\alpha, U_\alpha \dots$ be a play of $G^p_o(\kappa)$, where player two plays according to $\tau$. Then $U_\alpha \in \sigma((\tau((x_\gamma: \gamma \leq \beta): \beta < \alpha))$, for every $\alpha < \kappa$ and hence, since $\sigma$ is a winning strategy for player I in $G^\kappa_1(\mathcal{O}, \mathcal{O}_D)$ we must have that $\bigcup_{\alpha < \kappa} U_\alpha$ is not dense, and we are done.

\end{proof}

We are now going to exploit this result to give a short proof to a result of Angelo Bella from \cite{BS}.

\begin{theorem} \label{thmang}
Let $(X, \tau)$ be a regular space. Suppose that player two has a winning strategy in $G^{\kappa}_1(\mathcal{O}, \mathcal{O}_D)$. Then $d(X) \leq \chi(X)^{<\kappa}$
\end{theorem}

\begin{proof}
Using Theorem $\ref{thmequiv}$ fix a winning strategy $\sigma$ for player one in $G^p_o(\kappa)$. Let $M$ be a $<\kappa$-closed elementary submodel of $H(\theta)$, for large enough regular $\theta$ such that $|M|=\chi(X)^{<\kappa}$, $(X,\tau), \sigma \in M$ and $\chi(X) +1 \subset \kappa$.

We claim that $X \cap M$ is dense in $X$. Suppose not, and let $V \subset X$ be an open set such that $\overline{V} \cap X \cap M=\emptyset$. For every $x \in X \cap M$, let $\mathcal{U}_x \in M$ be a local base of size $\leq \chi(X)$. Since $\chi(X)+1 \subset M$ we have $\mathcal{U}_x \subset M$, and hence we can find in $M$ a neighbourhood  $U_x$ of $x$ such that $U_x \cap V=\emptyset$.

Since we have both $X \in M$ and $\sigma \in M$, the first move of player one $\sigma((\emptyset)):=x_0$ is a point of $M$. Let player two pick the open set $U_{x_0}$

Suppose that for some countable ordinal $\alpha$, player two picked open set $U_{x_\beta}$ at inning $\beta$ for every $\beta < \alpha$. Note that $\{U_{x_\beta}: \beta < \alpha \} \subset M$ and since $M$ is $<\kappa$ closed we have $\{U_\beta: \beta < \alpha \} \in M$. Therefore $x_\alpha:=\sigma((X)^\frown (U_\beta: \beta < \alpha)) \in M$. Let player two play $U_{x_\alpha}$ at the $\alpha$th inning. Since $\sigma$ is a winning strategy for player one, we must have that $\bigcup \{U_{x_\alpha}: \alpha < \omega_1 \}$ is a dense set. But this is impossible, because $V \cap U_{x_\alpha}=\emptyset$ for every $\alpha < \kappa$.
\end{proof}

\begin{corollary}
(A. Bella) Let $X$ be a first-countable regular space. If player two has a winning strategy in $G^{\aleph_1}_1(\mathcal{O}, \mathcal{O}_D)$ then $|X| \leq 2^{\aleph_0}$.
\end{corollary}

\begin{proof}
Every first countable space with a dense set of cardinality continuum has cardinality at most $(2^{\aleph_0})^{\aleph_0}=2^{\aleph_0}$.
\end{proof}

\begin{corollary}
Every first countable regular space where player II has a winning strategy in $G^\omega_1(\mathcal{O}, \mathcal{O}_D)$ is separable.
\end{corollary}

\begin{theorem}
Let $X$ be a first countable regular space. If player two has a winning strategy in $G^{\aleph_0}_1(\mathcal{O}, \mathcal{O}_D)$ then $X$ is separable.
\end{theorem}

The above theorem would lead one to conjecture that if player two has a winning strategy in $G^{\aleph_1}_1(\mathcal{O}, \mathcal{O}_D)$ on a first countable regular space $X$, then $X$ should have density $\aleph_1$, but in Example $\ref{exdensity}$ we are going to show that this is not the case, not even if  $G^{\omega_1}_1(\mathcal{O}, \mathcal{O}_D)$ is replaced with the stronger (for player II) game $G^{\omega_1}_1(\mathcal{O}, \mathcal{O})$.

\begin{lemma} \label{lemmaex}
Assume $cov(\mathcal{M})>\aleph_1+ \mathfrak{c}=\aleph_2$. Then there is a set $Y \subset \mathbb{I}=[0,1]$ of cardinality $\aleph_2$ such that the intersection of $Y$ with every meager set of $\mathbb{I}$ has cardinality at most $\aleph_1$.
\end{lemma}

\begin{proof}
$MA_{\omega_1}$ implies that $\mathbb{I}$ is not the union of $\aleph_1$ many nowhere dense sets.
Note that $\mathbb{I}$ has continuum many closed sets, so we can use $\mathfrak{c}=\omega_2$ to fix an enumeration $\{F_\alpha: \alpha < \omega_2 \}$ of the closed nowhere dense subsets of $\mathbb{I}$. We are going to construct $Y$ by transfinite induction. Suppose we have chosen points $\{y_\alpha: \alpha < \beta \} \subset \mathbb{I}$, where $\beta < \omega_2$. Choose any point $y_\beta \in \mathbb{I} \setminus (\bigcup_{\alpha < \beta} F_\alpha \cup \{y_\alpha: \alpha < \beta \})$. Then $Y=\{y_\alpha: \alpha < \omega_2\}$ is the desired set.
\end{proof}

\begin{example} \label{exdensity}
($cov(\mathcal{M})>\aleph_1+\mathfrak{c}=\omega_2$) A first countable regular space $X$ such that player two has a winning strategy in $G^{\omega_1}_1(\mathcal{O}, \mathcal{O})$ and $d(X)>\aleph_1$.
\end{example}

\begin{proof}
Recall the construction of the Alexandroff Double $\mathbb{D}$ of the unit interval. We define a topology on $\mathbb{D}=\mathbb{I} \times \{0,1\}$ by declaring every point of $\mathbb{I} \times \{1\}$ to be isolated and declaring a neighbourhood of a point $(x,0) \in \mathbb{I} \times \{0\}$ to be of the form $U \times \{0\} \cup U \times \{1\} \setminus F$, where $U$ is an Euclidean open set and $F$ is a finite set. It is easy to see that $\mathbb{D}$ is a compact Hausdorff space with points $G_\delta$, and hence it's first countable. 

Let now $Y$ be the set constructed in Lemma $\ref{lemmaex}$. Without loss of generality we can assume that $Y$ is dense in $\mathbb{I}$. Now, consider the space $X= Y \times \{0,1\}$ with the topology induced by $\mathbb{D}$. Then $X$ is a regular first countable space of density $\aleph_2$. Fix a countable dense set $D \subset Y$. We let $C=D \times \{0\}$.

\noindent {\bf Claim.} The complement of every open set containing $C$ has cardinality at most $\aleph_1$.

\begin{proof}[Proof of Claim] Let $U$ be an open subset of $X$ such that $C \subset U$. Since $Y$ is hereditarily Lindelof, we can assume that $U=\bigcup_{n<\omega} (U_n \times \{0\} \cup U_n \times \{1\} \setminus F_n)$, where $U_n$ is the trace on $Y$ of an Euclidean open set and $F_n$ is finite. Let $O=\bigcup_{n<\omega} U_n$ and let $V$ be an open subset of $\mathbb{I}$ such that $V \cap Y=O$. Since $V$ is dense in $\mathbb{I}$,  $\mathbb{I} \setminus V$ is nowhere dense, and hence $A=(\mathbb{I} \setminus V) \cap Y$ has cardinality at most $\aleph_1$. Now $X \setminus U \subset (A \times \{0,1\}) \cup \bigcup_{n<\omega} F_n$ and since the latter set has cardinality at most $\aleph_1$, also $X \setminus U$ has cardinality at most $\aleph_1$, as we wanted.
\renewcommand{\qedsymbol}{$\triangle$}
\end{proof}

A winning strategy for player two in $G^{\omega_1}_1(\mathcal{O}, \mathcal{O})$ is now easy to describe. Let $\{x_n: n < \omega \}$ be an enumeration of $C$, and suppose that $\mathcal{U}_\alpha$ is the open cover played by player one at inning $\alpha$. At inning $n<\omega$ player two picks an open set $U_n \in \mathcal{U}_n$ such that $x_n \in U_n$. Let $\{z_\alpha: \alpha < \omega_1\}$ be an enumeration of $X \setminus \bigcup_{n<\omega} U_n$. Then at inning $\omega+\alpha$ player two simply picks an open set $U_{\omega+\alpha} \in \mathcal{U}_{\omega+\alpha}$ such that $z_\alpha \in U_{\omega+\alpha}$. We then have that $\{U_\alpha: \alpha < \omega_1 \}$ is an open cover, regardless of player one's choices, and hence the strategy we have defined is a winning one.
\end{proof}

\begin{question}
Can we get a space with the features of Example $\ref{exdensity}$ simply from the negation of CH?
\end{question}

Here is another application of Theorem $\ref{thmequiv}$, regarding the behavior of the game version of weak Lindel\"ofness in products.

\begin{theorem}
Suppose player two has a winning strategy in $G^\omega_1(\mathcal{O}, \mathcal{D})$ on $X$ and $Y$ is separable. Then player two has a winning strategy in $G^\omega_1(\mathcal{O}, \mathcal{D})$ on $X \times Y$.
\end{theorem}

\begin{proof}
Exploiting Theorem $\ref{thmequiv}$ fix a winning strategy $\sigma_X$ for player one on $X$ in $G^p_o(\omega)$ and fix a countable dense set $\{d_n: n < \omega \}$ for $Y$. Partition $\omega$ into infinitely many sets $\{I_k: k < \omega \}$ in such a way that if $\{i^k_j: j < \omega \}$ is an increasing enumeration of $I_k$ then $\{i^k_j: k < \omega \}$ is an increasing sequence, for every $j<\omega$. We now define a winning strategy for player one in $G^p_o(\omega)$ .

Assume player two played open set $U_i \times V_i$ at inning $i$, for $i \leq n$, and let $k, j< \omega$ be such that $n=i^k_j$. Let $\sigma_{X \times Y}((U_1 \times V_1, U_2 \times V_2, \dots U_n \times V_n))$ be $(x_n,y_n)$, where $x_n=\sigma_X((U_{i^k_m}: m \leq j ))$ and $y_n=d_k$.

\noindent Let now $$(\sigma_{X \times Y}(\emptyset), U_1 \times V_1, \pi_{X \times Y}(U_1\times V_1), \dots)$$ be a play. We claim that $\bigcup \{U_n \times V_n: n < \omega \}$ is dense. Indeed, let $U \times V$ be a non-empty basic open set in $X \times Y$. Then there is $k<\omega$ such that $d_k \in V$ and since $\pi_X$ is a winning strategy for $G^p_o(\omega)$ on $X$ then there is $n \in I_k$ such that $U \cap U_n \neq \emptyset$. Since $d_k \in V_n$ we have $(U \times V) \cap (U_n \times V_n) \neq \emptyset$ and hence we are done.
\end{proof}

Let's recall the definition of the point open game $G^o_p(\kappa)$, due to Berner and Juh\'asz \cite{BJ}.

\begin{definition}
Two players play $\kappa$ many innings. At inning $\alpha$, player one chooses an open set $O_\alpha$ and player two plays a point $x_\alpha \in O_\alpha$. Player one wins in $G^o_p(\kappa)$ if $\{x_\alpha: \alpha < \kappa \}$ is dense in $X$.
\end{definition}

Scheepers proved in \cite{SchSep} that the game $G^o_p(\kappa)$ is equivalent to the generalized $R$-separability game $G^\kappa_1(\mathcal{D}, \mathcal{D})$, in the following sense.

\begin{theorem} \label{thmequiv2} {\ \\}
\begin{enumerate}
\item \label{Splayerone} Player one has a winning strategy in $G^o_p(\kappa)$ if and only if player two has a winning strategy in $G^\kappa_1(\mathcal{D}, \mathcal{D})$.
\item \label{Splayertwo} Player two has a winning strategy in $G^o_p(\kappa)$ if and only if player one has a winning strategy in $G^\kappa_1(\mathcal{D}, \mathcal{D})$.
\end{enumerate}
\end{theorem}

He actually stated the result only for the case $\kappa=\omega$ 

We now exploit Scheeper's result to prove that every regular space where player two has a winning strategy in $G^{\omega_1}_1(\mathcal{D}, \mathcal{D})$ has $\pi$-weight at most continuum. As a byproduct we obtain an alternative proof of Scheeper's result countable $\pi$-weight is equivalent to the property that player two has a winning strategy in the $R$-separability game (see \cite{SchSep}).

\begin{theorem}
Let $X$ be a regular space and suppose that player two has a winning strategy in $G^{\kappa}_1(\mathcal{D}, \mathcal{D})$. Then $\pi w(X) \leq 2^{<\kappa}$.
\end{theorem}

\begin{proof}
Let $\tau$ be the set of all open sets of $X$ and fix a strategy $\sigma$ for player one in $G^o_p(\kappa)$. Let $M$ be a $<\kappa$-closed elementary submodel of $H(\theta)$, for some large enough regular $\theta$, such that $X, \tau \in M$, $|M| \leq 2^{<\kappa}$.

\noindent {\bf Claim.} $X \cap M$ is dense in $X$.

\begin{proof}[Proof of Claim]
Play a game of $G^o_p(\omega_1)$ where the first player plays according to $\sigma$ and the second player picks all its points in $M$. In other words, let $\alpha<\kappa$, and suppose that at inning $\beta < \alpha$, player one picked non-empty open set $U_\beta \in M$ and player two picked a point $x_\beta \in U_\beta \cap M$. At inning $\alpha$, player one plays non-empty open set $U_\alpha=\sigma((x_\beta: \beta < \alpha))$ which is an element of $M$ by $\kappa$-closedness and player two plays any point $x_\alpha \in U_\alpha \cap M$. Since $\sigma$ is a winning strategy for player one we must have that $\{x_\alpha: \alpha <\kappa\}$ is a dense set. But $\{x_\alpha: \alpha < \kappa \} \subset X \cap M$, and hence the claim.
\renewcommand{\qedsymbol}{$\triangle$}
\end{proof}

 We now claim that $\tau \cap M$ is a $\pi$-base. Suppose this is not the case, and let $V$ be an open set such that $U \nsubseteq V$ for every $U \in \tau \cap M$. By regularity of $X$ we can actually assume that $U \nsubseteq \overline{V}$, for every $U \in \tau \cap M$. Now play a game of $G^o_p(\kappa)$ in a similar way as in the above claim, with the only difference that at inning $\alpha$ player two picks a point $x_\alpha \in U_\alpha \setminus \overline{V} \cap M$. Since player one is playing according to $\sigma$, we again have that $\{x_\alpha: \alpha < \kappa \}$ is dense, but this is impossible, since $V \cap \{x_\alpha: \alpha < \kappa \} = \emptyset$.
\end{proof}

\begin{corollary}
(Scheepers) Let $X$ be a regular space. Then $X$ has a countable $\pi$-base if and only if player two has a winning strategy in $G^\omega_1(\mathcal{D}, \mathcal{D})$.
\end{corollary}

\section{Acknowledgements}
The first listed author was partially supported by FAPESP (2013/05469-7) and by 
Austrian Science Funds (FWF) through the grant M1244-N13. He wishes to thank KGRC fellows 
for the great hospitality. The second named author was partially supported by FAPESP 
postdoctoral grant 2013/14640-1,
 \emph{Discrete sets and cardinal invariants in set-theoretic topology}. 
Part of the work for the paper was carried out when he visited the third author
 at the KGRC in Vienna and when he visited the first author at the ICMC in Sao Carlos. 
He wishes to thank his colleagues in Vienna and Sao Carlos for their warm hospitality and 
the University of Vienna and FAPESP for financial support.
The third-listed author  thanks the Austrian Academy of Sciences and the
 Austrian Science Funds (FWF) for  generous support through the APART Program 
and grants I 1209-N25, M 1244-N13, respectively.

\end{document}